\documentclass[12pt]{amsart}

\usepackage{amsmath, amssymb, latexsym, amsthm, color}
\usepackage[numbers]{natbib} 
\usepackage{multicol} 

\usepackage{physics}
\usepackage{braket}
\usepackage{xparse}

  \NewDocumentCommand{\tens}{t_}
 {%
  \IfBooleanTF{#1}
   {\tensop}
   {\otimes}%
 }
\NewDocumentCommand{\tensop}{m}
 {%
  \mathbin{\mathop{\otimes}\displaylimits_{#1}}%
 }

\def\N{\mathbb{N}}

\DeclareMathOperator{\Span}{Span}
\DeclareMathOperator{\dep}{depth}

\newtheorem{defn}{Definition}[section]
\newtheorem{prop}[defn]{Proposition}

\newtheorem{rem}[defn]{Remark}
\newtheorem{nota}[defn]{Notation}
\newtheorem{thm}[defn]{Theorem}
\newtheorem{lem}[defn]{Lemma}
\theoremstyle{definition}

\begin{document}

\title{A generalized infinite quantum Ramsey theorem for operator systems}
\author{Jos\'e G. Mijares}
\address{California State University Los Angeles, Department of Mathematics}
\email{jmijare5@calstatela.edu}

\begin{abstract}
We prove a generalization of the infinite quantum Ramsey theorem of Kennedy et al. \cite{kks}, showing that it follows from an  archetypical ``selective" pattern satisfied by certain families of projections in an infinite-dimensional Hilbert space.  
\end{abstract}

\maketitle

\section{Preliminaries}\label{prelim}

Recently, Weaver \cite{weav17} proved a result about operator systems in finite-dimensional Hilbert spaces that is understood as a quantum analog to the finite version of Ramsey's theorem \cite{ram}.  The word ``quantum" is used here to indicate that operator systems play the role of the confusability graph of quantum channels. That is why operator systems are referred to as a ``quantum graphs" in the context of quantum information theory. Later, Kennedy et al. \cite{kks} proved an infinite-dimensional version of Weaver's result which is understood as quantum generalization of the infinite version of Ramsey's theorem. In this article, we show that this latter result is an instance of a more general pattern that can be deduced from the characteristic ``selective" behavior of certain families of projections in infinite-dimensional Hilbert spaces.

We characterized this general pattern by studying selective families of infinite-rank projections which contain diagonal projections induced by selective coideals on $\mathbb N$. See Definitions \ref{wide def} and \ref{sele def} below.

\subsection{Operator systems as quantum graphs}

The following definitions are taken from \cite{weav21}. These are understood as ``quantum" analogs of the idea of a relation on a set, etc. Let $\mathcal H$ be a complex Hilbert space and let $ B(\mathcal H)$ be the bounded operators on $\mathcal H$. Equip $ B(\mathcal H)$ with the weak topology: a net $(T_n)$ in $ B(\mathcal H)$  converges weakly to $T\in B(\mathcal H)$ if and only if $\lim_n \langle T_n(u), v \rangle =  \langle T(u), v \rangle $, for all $u,v \in\mathcal H$.

\begin{defn}
\begin{enumerate}
\item A \textbf{quantum relation} on $\mathcal H$ is a subspace $\mathcal V\subseteq B(\mathcal H)$.
\item If $I_{\mathcal H}$ is the identity operator on $\mathcal H$, then $\mathbb C I_{\mathcal H} = \{zI_{\mathcal H} : z\in \mathbb C\}$ is the \textbf{quantum diagonal} relation on $\mathcal H$.
\item Let $\mathcal V\subseteq \mathcal H$ be a quantum relation. For every $A\in\mathcal V$,  $A^{*}$ denotes the adjoint of $A$. The set $\mathcal V^{*} = \{A^{*} : A\in\mathcal V\}$ denotes the \textbf{adjoint} of $\mathcal V$.
\end{enumerate}
\end{defn}

\begin{defn} A quantum relation $\mathcal V$ is:
\begin{enumerate}
\item \textbf{reflexive}, if $\mathbb C I_{\mathcal H}\subseteq \mathcal V$ (i.e., $ \mathcal V$ is \textit{unital}).
\item \textbf{symmetric},  if $\mathcal V^{*}=\mathcal V$ (i.e., $ \mathcal V$ is \textit{self-adjoint}).
\item a \textbf{quantum graph}, if it is reflexive, symmetric.
\end{enumerate}
\end{defn}

\begin{rem}
Note that quantum graphs as defined above are usually referred to as \textbf{operator systems}. The reason for using the term `quantum graph' has to do with the fact that, in Quantum Information Theory,  operator systems play a role (for quantum channels) that is analogous to the role played by `confusability graphs', for classical channels (see \cite{dsw, stah, weav21}, for instance). Also note that even though a simple (classical) graph is not reflexive (has no loops) its structure can be obviously encoded in a reflexive and symmetric relation which somehow motivates the quantum-graph analogy.  This graph-theoretic terminology is also helpful to intuitively introduce and understand Ramsey-like concepts in the quantum (non commutative) context. Ramsey theory, combinatorics and graph theory have seen applications to classical information theory before. See for instance the work of Lov\'asz \cite{lov}, where a combinatorial study of the Shannon capacity problem is done, from the point of view of graph theory. 
\end{rem}

\subsection{Quantum Ramsey theorems. Ramsey families of projections}

If $P\in B(\mathcal H)$ is an orthogonal projection (i.e., $P^2 = P^* = P$) and $\mathcal{V}\subseteq B(\mathcal H)$  is a quantum relation, then 
 \[P\mathcal{V}P = \{PAP : A\in \mathcal{V}\}\] is the \textbf{compression} of $\mathcal{V}$ by $P$. 

\begin{defn} 
(\cite{weav17}) Let $k$ and $n$ be a positive integers and let $\mathcal{V}\subseteq M_n(\mathbb C)$  be a quantum graph.  A \textbf{quantum $k$-clique} for $\mathcal{V}$ is an  orthogonal projection $P\in M_n(\mathbb C)$ of rank $k$ such that $P\mathcal{V}P\cong M_k(\mathbb C)$ (equivalently, $\dim P\mathcal{V}P = k^2$). A  \textbf{quantum $k$-anticlique} for $\mathcal{V}$   is an  orthogonal projection $P\in M_n(\mathbb C)$ of rank $k$ such that $P\mathcal{V}P\cong \mathbb C P$ (equivalently, $\dim P\mathcal{V}P = 1$).  
\end{defn}

In \cite{weav17}, Weaver proves the following result which can be seen as a quantum version of the finite Ramsey theorem:

\begin{thm}[Quantum Finite Ramsey Theorem for Operator Systems; Weaver \cite{weav17}, Theorem 4.5]
Let $k$ be a positive integer. For any $n\geq 8k^{11}$, every operator system $\mathcal V\subseteq M_n$ has either a quantum
$k$-clique or a quantum $k$-anticlique.
\end{thm}

\begin{defn}

(\cite{kks}) Let $\mathcal{V}\subseteq B(\mathcal H)$  be a quantum graph in an infinite-dimensional Hilbert space $\mathcal H$. An \textbf{infinite quantum clique} for $\mathcal{V}$ is an infinite-rank orthogonal projection $P\in B(\mathcal H)$ such that $P\mathcal{V}P$ is weakly dense in $P B(\mathcal H) P$. An  \textbf{infinite quantum anticlique} for $\mathcal{V}$   an infinite-rank orthogonal projection $P\in B(\mathcal H)$ such that $P\mathcal{V}P \cong \mathbb C P$. An infinite-rank orthogonal projection $P$ is an \textbf{infinite quantum obstruction} for $\mathcal{V}$ if  $P\mathcal{V}P$ is finite-dimensional, every operator in   $P\mathcal{V}P$ can be written as the sum of a scalar multiple of $P$ and a compact operator, and $P\mathcal{V}P$ contains a self-adjoint compact operator with range dense in $P\mathcal{H}$. 
\end{defn}

\begin{rem}
By  the Knill–Laflamme error-correction conditions, quantum anticliques correspond to error-correcting codes for quantum channels (see \cite{stah}).
\end{rem}
 
\begin{thm}[Kennedy et al.,  \cite{kks}, Theorem 5.1]\label{kks}
Every operator system $\mathcal V\subseteq B(\mathcal H)$ on an infinite
dimensional Hilbert space $\mathcal H$  has either an infinite quantum
clique, an infinite quantum anticlique or an infinite quantum obstruction.
\end{thm}

Besides the potential applications of the above results to quantum error-correction theory, they provide an interesting context to which Ramsey theory can be extended. This fact motivates the research in this article.

Let $\Pi_{\infty}(\mathcal H)\subset B\left(\mathcal{H}\right)$ denote the family of infinite-rank orthogonal projections on $\mathcal H$. The above results motivate the following:

\begin{defn}
 A family of projections $\mathcal F\subseteq\Pi_{\infty}(\mathcal H)$ is \textbf{Ramsey} if for every operator system $\mathcal{V}\subseteq B\left(\mathcal{H}\right)$ there exists $P\in\mathcal F$ satisfying the following trichotomy: 
 \begin{enumerate}
\item Either $P$ is an infinite quantum clique for $\mathcal{V}$.
\item  $P$ is infinite quantum anticlique for $\mathcal{V}$, or 
\item $P$ is infinite quantum obstruction for $\mathcal{V}$. 
\end{enumerate}
	 \end{defn}
	 
Thus, Theorem \ref{kks} above states that the family $\Pi_{\infty}(\mathcal H)$ is Ramsey. It is worth mentioning that a "Ramsey property" is usually stated as a dichotomy, and there is a way to state the definition above as such (see Definition 1.4 in \cite{kks}) by combining being an infinite quantum anticlique and being an infinite quantum obstruction in one single statement: $P$ is \textbf{infinite quantum almost anticlique} for $\mathcal{V}$ if $P\mathcal V P$ is finite-dimensional and every operator in  $P\mathcal V P$  can be written as the sum of a scalar multiple of $P$ and a compact operator. But we decided to state the definition without combining these two conditions because it would make more clear the potential connections between Ramsey theory and quantum information theory (quantum error-correction) via the possibility of having an infinite quantum anticlique and thus satisfying the Knill–Laflamme error-correction conditions (see \cite{stah}).  

As shown in our main result, Theorem \ref{selective is Ramsey} below, what guarantees the ``Ramseyness" of $\Pi_{\infty}(\mathcal H)$ is a natural ``selective" property that it satisfies (see Definition \ref{sele def} below). In fact, Theorem \ref{selective is Ramsey} gives a characterization of the families $\mathcal F\subseteq \Pi_{\infty}(\mathcal H)$ which have this selective property and contain, essentially, all the diagonal projections. 

In Section \ref{widefam} we provide some facts about families containing the diagonal projections, including a version of the Dilation Lemma of \cite{kks}. In Section \ref{selefam} we formally introduce the selective families vaguely mentioned above, and prove some results concerning diagonalizable operator systems. In Section \ref{main}, after proving facts about wide, selective families and diagonalizable operator systems, we prove our main result.


\section{Wide families of projections via coideals on $\mathbb N$}\label{widefam}

Throughout the rest of this article, we will fix a countably infinite-dimensional Hilbert space $\mathcal H$ and an orthonormal basis $\mathfrak B = \{h_n\}$  for $\mathcal H$.


Let $\mathbb N^{[\infty]} = \{A\subseteq \mathbb N : |A|=\infty\}$. For $A\in\mathbb N^{[\infty]}$, let 

\[P^{\mathcal B}_A := \mbox{ the orthogonal projection onto } \overline{\Span}\{h_n : n\in A\}.\]

Recall the \textbf{diagonal embedding} 

\[\mathbb N^{[\infty]} \longrightarrow \Pi_{\infty}(\mathcal H)\] 
\[A \rightarrow P^{\mathcal B}_A.\]

$P^{\mathcal B}_A$ is said to be a \textbf{diagonal projection}. 

\begin{defn}
A family $\mathcal C\subseteq\mathbb N^{[\infty]}$ is a \textbf{coideal}, if it satisfies the following:

\medskip

For $A,B\in\mathbb N^{[\infty]}$,
\begin{enumerate}
\item If $A\in\mathcal C$ and $A\subseteq B$, then $B\in\mathcal C$.
\item If $A\cup B\in\mathcal C$, then $A\in\mathcal C$ or $B\in\mathcal C$.
\item If $A\in\mathcal C$ and $|A\bigtriangleup B|<\infty$, then $B\in\mathcal C$.
\end{enumerate}
\end{defn}

We will use coideals to build families of projections that satisfy the Ramsey property defined above, yielding analogs of the infinite quantum Ramsey theorem of \cite{kks}. Let $\wp(\mathbb N)$ denote the power set of $\mathbb N$. The complement $\mathcal I = \wp(\mathbb N) \setminus\mathcal C$ of a coideal is an \textbf{ideal} of sets. The study of ideals, coideals, their properties and their connections with set theory, consistency, Ramsey theory, and other applications, has been extensive.  See for instance \cite{dmu, far, hmtu, mathias, sole, tod}.

\begin{nota}
Given a coideal $\mathcal C\subseteq\mathbb N^{[\infty]}$, let \[\mathcal F_{\mathcal C} = \{P^{\mathcal B}_A : A\in\mathcal C\}.\]

Note that, by the definition of coideal, if $I_{\mathcal H}$ is the identity on $\mathcal H$, then \[I_{\mathcal H} = P^{\mathcal B}_{\mathbb N}\in\mathcal F_{\mathcal C}.\] 
\end{nota}

The following definition will be used in the next section. 
\begin{defn}
A coideal $\mathcal C\subseteq\mathbb N^{[\infty]}$ is \textbf{selective}, if for every decreasing sequence $A_1\supseteq A_2\supseteq \dots$ in $\mathcal C$, there exists $B\in C$ such that $B/n := \{k \in B : k>n\}\subseteq A_n$, for all $n\in B$.
\end{defn}

Now consider the following: Let $T\in\mathcal B(\mathcal H, \ell^2(\mathbb N))$ be an infinite-rank operator, and let \[A_T = \{n : T(h_n)\neq \overrightarrow{0}\}\in\mathbb N^{[\infty]}.\] Notice that

\[\ker(T)^{\perp} = \overline{\Span}\{h_n : n\in A_T\}.\] Therefore, if $P_{\ker(T)^{\perp}}$ is the orthogonal projection onto $\ker(T)^{\perp}$ then, 

\[P_{\ker(T)^{\perp}} = P^{\mathcal B}_{A_T}\in \mathcal F_{\mathbb N^{[\infty]}}. \]

Moreover, given a coideal $\mathcal C\subseteq\mathbb N^{[\infty]}$ and $C\in\mathcal C$ such that $|\mathbb N \setminus C| = |\mathbb N \setminus A_T|$, let $\varphi: C\rightarrow A_T$ and  $\psi: \mathbb N \setminus C\rightarrow \mathbb N \setminus A_T$ be bijections. Define the isometry $U:\mathcal H\rightarrow \mathcal H$ as

\[U(h_n)=\left\{ \begin{array}{rcl}
h_{\varphi(n)} & \mbox{if} & n\in C\\
& & \\
h_{\psi(n)} & \mbox{if} &  n\in\mathbb N \setminus C\\
\end{array}\right.\]

and extend it linearly to all $v\in\mathcal H$. Also, let  $\hat{T} = T\circ U$. Then,

\[P_{\ker(\hat{T})^{\perp}} = P^{\mathcal B}_{C}\in \mathcal F_{\mathcal C}. \]

This leads to the following definition. 

\begin{defn}\label{wide def}
A family $\mathcal F$ of infinite-rank projections is \textbf{wide} if there exists a coideal $\mathcal C\subseteq\mathbb N^{[\infty]}$ such that $\mathcal F_{\mathcal C}\subseteq \mathcal F$.
\end{defn}

The following is an application of the Dilation Lemma of \cite{kks}.
 
\begin{lem}\label{wide}
 Let $\mathcal F \subseteq\Pi_{\infty}(\mathcal H)$ be a wide family of projections and let  $\mathcal V\subseteq B(\mathcal H)$ be an operator system. Let $T\in B(\mathcal H, \ell^2(\mathbb N))$ be an operator such that $T\mathcal VT^*$ is weakly dense in $B( \ell^2(\mathbb N))$. Then $\mathcal F$ contains an infinite quantum clique for $\mathcal V$.
\end{lem}

\begin{proof}
Let $\mathcal C\subseteq\mathbb N^{[\infty]}$ be a coideal such that $\mathcal F_{\mathcal C}\subseteq \mathcal F$. By Lemma 2.1 of \cite{kks}, $P := P_{\ker(T)^{\perp}}$ is an infinite quantum clique for $\mathcal V$. If $\mathcal C = \mathbb N^{[\infty]}$ then $P \in\mathcal C$ and we are done. Otherwise, define  $\hat{T}\in B(\mathcal H, \ell^2(\mathbb N))$ as above. Then $Q := P_{\ker(\hat{T})^{\perp}}\in\mathcal C$. Note that $Q = U^{-1}P$. So, 
\[Q\mathcal V Q =  Q\mathcal V Q^* = U^{-1} P\mathcal V P (U^{-1})^* \]
Since $P\mathcal V P$ is weakly dense in $PB( \mathcal H)P$, then $Q\mathcal V Q$ is weakly dense in \[U^{-1} PB( \mathcal H) P(U^{-1})^* = Q B( \mathcal H) Q.\]  This completes the proof.
\end{proof}


\section{Selective families of projections}\label{selefam}

As mentioned, though vaguely, in Section \ref{prelim}, the family $\Pi_{\infty}(\mathcal H)$ satisfies a selectivity property that guarantees the Ramsey-like behavior of operator systems. In this section, we will give a formal definition of that property, inspired by similar concepts from infinite combinatorics and set theory (see \cite{cdm,dmn, far, goyoMLQ07, mathias, tod}, for instance). We will also prove some useful facts about families of projections satisfying that property.


\begin{nota}  We will adopt the following notation:

\begin{enumerate}

\item If $v\in \mathcal H$ is a vector, define \[\ \ \ \ \ \dep(v) = \min\{n\in\N : v\in\Span\{h_1, \dots, h_n\}\}\]
 \item   Given a projection $P\in  \Pi_{\infty}(\mathcal H)$, for every $v\in \mathcal H$ let 
 \[P/v : = \mbox{ the projection onto }  \overline{\Span}\{P h_k\in P\mathcal H\ : k >\dep(v)\}.\]  
And for $n\in \mathbb N$ , let \[P/n := P/h_n.\] 
  \item Given projections $P$ and $Q$, we write $P\le Q$ if and only if $P\vee Q =  Q$. 

\item  For a set of vectors $S\subseteq \mathcal H$, denote by $P_S$ the projection onto the closure of $\Span S$.  
\item  For a family $\mathcal F\subseteq \Pi_{\infty}(\mathcal H)$, let $[\mathcal F]$ denote the collection of all closed subspaces $E$ of $\mathcal H$ such that $P_E\in \mathcal F$.

\item If $E$ and $F$ are closed subspaces of $\mathcal H$, let $E\ominus F := E\cap F^{\perp}$.

 \end{enumerate}
\end{nota}

\medskip
\begin{defn}\label{sele def} A family $\mathcal F\subseteq \Pi_{\infty}(\mathcal H)$ of infinite-rank projections is  \textbf{selective} if it satisfies the following: 
\begin{enumerate}
	\item For every decreasing sequence $\{P_n\}_n$ in $\mathcal F$ (i.e., for every $n, P_n\in\mathcal F$ and $P_{n+1}\le P_n$), there exists $P \in\mathcal F$ such that $P/n\leq P_n$ for all $n$. 
			\item For $P, Q\in  \Pi_{\infty}(\mathcal H)$,  if $P\vee Q \in \mathcal F$, then either $P\in \mathcal F$ or $Q\in \mathcal F$.
			\item  $\mathcal F$ is ``closed under finite changes":  For $P\in\mathcal F$,
			\begin{enumerate}
			\item If $F$ is a finite-dimensional  subspace of $P\mathcal H$ and $Q$ is the orthogonal projection onto $P\mathcal H\ominus F$, then $Q\in\mathcal F$. 
			\item If $F$ is a finite-dimensional  subspace of $(P\mathcal H)^{\perp}$ and $Q$ is the orthogonal projection onto $P\mathcal H\oplus F$, then $Q\in\mathcal F$.
	\end{enumerate}
		
	 \end{enumerate}
	 \end{defn}
	 
	Obviously, if  $\mathcal F  = \Pi_{\infty}(\mathcal H)$  then $\mathcal F$ is selective (and wide). The following result is also straightforward:

\begin{thm}
For every selective coideal $\mathcal C\subseteq\mathbb N^{[\infty]}$, the family $\mathcal F_{\mathcal C} $ is a wide, selective family of projections.
\end{thm}
\qed
 
In Theorem \ref{selective is Ramsey} below, we prove that every wide, selective family $\mathcal F\subseteq\Pi_{\infty}(\mathcal H)$ of infinite-rank projections is Ramsey. This is a natural generalization of the main result in \cite{kks}.
	
	\begin{lem}\label{one-two}  Let $\mathcal{V}\subseteq B\left(\mathcal{H}\right)$ be an operator system and let $\mathcal F \subseteq\Pi_{\infty}(\mathcal H)$ be a wide, selective family.  Then there exists $P\in\mathcal F$  satisfying one of the following conditions: 
	 \begin{enumerate}
\item  For every $n$,  $P/n\mathcal{H}\ominus P\mathcal V P h_n\notin [\mathcal F]$, or
	\item The compression,  $P\mathcal{V}P$ is diagonalizable. 
	\end{enumerate}
	\end{lem}

	\begin{proof}
	Assume that for every $P\in\mathcal F$, Condition (1) does not hold. Let us define a decreasing sequence $(P_n)$ in $\mathcal F$, and an increasing sequence $(k_n)$ of positive integers, as follows:
	
	Let $P_1 = I_{\mathcal H}\in\mathcal F$. Choose $k_1$ such  $P_1/k_1\mathcal{H}\ominus P_1\mathcal{V}P_1 h_{k_1}\in [\mathcal F]$. \\
	Let $P_2$ be the projection onto $P_1/k_1\mathcal{H}\ominus P_1\mathcal{V}P_1 h_{k_1}$. (That is, $P_2$ is the projection onto $\mathcal H\ominus \mathcal V  h_{k_1}$). By assumption, $P_2\in\mathcal F$. And $P_2\leq P_1$ by construction.
	
	Now, suppose $P_1\geq P_2\geq \dots \geq P_n$ in $\mathcal F$ and $k_1 < k_2, \dots < k_{n-1}$ have been defined such that 
	
	\[P_i/k_i\mathcal{H}\ominus P_i\mathcal{V}P_i h_{k_i}\in [\mathcal F],\] 
	
	for $1\leq i < n$, and $P_n$ is the projection onto $P_{n-1}/k_{n-1}\mathcal{H}\ominus P_{n-1}\mathcal{V}P_{n-1}h_{k_{n-1}}$. By assumption, there exists $k_{n}$ such that $P_n/k_{n}\mathcal{H}\ominus P_n\mathcal{V} P_n h_{k_{n}}$. Let $P_{n+1}$ be the projection onto $P_n/k_{n}\mathcal{H}\ominus P_n\mathcal{V} P_n h_{k_{n}}$. Then $P_{n+1}\leq P_n$ by construction and $P_{n+1}\in\mathcal F$. Note that, if necessary, by  repeating this step a finitely-many times we can assume $k_n> k_{n-1}$. This completes the definition of the sequences  $(P_n)$ and $(k_n)$ .
	
	By selectivity, there exists  $P \in \mathcal F$ such that,  for every $n\geq 1$, $P/n\leq P_n$. 
	
	Given $n\geq 1$, note that  \[P/k_n \leq P/n\leq P_n.\] In fact, 
	
	\[P/k_n \leq P_n/k_n.\]

	Let $x_n = Ph_{k_n}$. Then, for every $A\in\mathcal{V}$:
	
	\[\langle P A Px_{n-1},  x_n\rangle = \langle P A P h_{k_{n-1}}, P h_{k_n}\rangle = \langle P A P h_{k_{n-1}},  h_{k_n}\rangle = \langle P_{n-1} A P_{n-1} h_{k_{n-1}},  h_{k_{n}}\rangle = 0.\] 
	
	In fact, for all $l,m\geq 1$  with $ l < m$ and every $A\in\mathcal{V}$:
	
	\[\langle P A P x_l, x_m\rangle = \langle P A P h_{k_l}, Ph_{k_m}\rangle = \langle P_l A P_l h_{k_l}, P_l h_{k_m}\rangle = \langle P_l A P_l h_{k_l} , h_{k_m}\rangle = 0.\] 
	
	Now, using the fact that $\mathcal V$ is self-adjoint, we conclude that for all $l, m\geq 1$ and every $A\in\mathcal{V}$:
	
	\[\langle P A P x_l, x_m\rangle = 0.\]  So,  $P\mathcal{V}P$ is diagonalizable.

		\end{proof}

\begin{lem}\label{diminf}
If $\mathcal V\subseteq B(\mathcal H)$ is an operator system such that for every $n$, $\dim \mathcal V h_n = \infty$, then there exists a sequence $(A_n)$ in $\mathcal V$ and an orthonormal sequence $(x_n)$ in $\mathcal H$ such that for each $n$, $\langle A_n x_n, x_{n+1}\rangle = 1$ and  $\langle A_n x_i, x_j\rangle = 0$, for $i\neq j$ with $\max\{i,j\}>n+1$. 
\end{lem}

\begin{proof}
If $\dim \mathcal V h_n = \infty$ for every $n$, then $\dim \mathcal V h = \infty$ for every nonzero $h$ in $\mathcal H$. So the result follows from Lemma 3.3 of \cite{kks}.

\end{proof}


\begin{lem}\label{compress} Let $\mathcal F \subseteq\Pi_{\infty}(\mathcal H)$ a wide, selective family. Let $\mathcal V\subseteq B(\mathcal H)$ be an operator system, and let $(A_n)$ be a linearly independent sequence of operators in $\mathcal V$. Let $(N_n)$, with $N_n\geq 1$, be a strictly increasing sequence of integers such that $\langle A_n h_i, h_j\rangle=0$ if $\max\{i,j\}>N_n$ and $i\neq j$.  Suppose there is no nonzero finite-rank operators in $\Span\{A_n : n\geq 1\}$. Then there exists $P\in\mathcal F$ such that the sequence $(PA_n P)$ of compressions is linearly independent and for $n\geq 1$, $\langle A_n P h_i, P h_j\rangle=0$ if $\max\{i^3,j^3\}>n$ and $i\neq j$. 
\end{lem}
\begin{proof}

 Since $A_1$ has infinite rank, there exists $M_1>N_1$ such that $\langle A_n h_{M_1}, h_{M_1}\rangle=0$. Since  $\mathcal F$ is wide and closed under finite changes, we can choose $P_1\in\mathcal F$ be such that $P_1 h_1 = h_{M_1}$. For $k\geq 2$, suppose $M_1, \dots, M_{k-1} \in \mathbb N$ and $P_1\geq P_2\geq  \dots \geq P_{k-1}$ in $\mathcal F$ have been chosen such that $\{P_i A_1P_i, \dots P_i A_{k-1}P_i\}$ is linearly independent, for all $i\in\{1, \dots, k-1\}$.

If there is an operator $A\in\Span\{A_1, \dots, A_{k-1}\}$ such that $P_{k-1} AP_{k-1} = P_{k-1} A_kP_{k-1}$, then, since $\{P_{k-1} A_1P_{k-1}, \dots P_{k-1} A_{k-1}P_{k-1}\}$ is linearly independent, $A$ is unique. Also, since $\{A_1, \dots, A_k\}$ is linearly independent, $A-A_k$ has infinite rank. So there is $M_k>N_{k^3}$ such that $\langle (A-A_k) h_{M_k}, h_{M_k}\rangle\neq 0$. Then, $\langle A h_{M_k}, h_{M_k}\rangle\neq \langle A_k h_{M_k}, h_{M_k}\rangle$. Otherwise, if no such $A\in\Span\{A_1, \dots, A_{k-1}\}$ exists, choose any $M_k>N_{k^3}$. In either case, let $P_k\in\mathcal F$ with $P_k\leq P_{k-1}$ be such that $P_k h_k = h_{M_k}$. It follows that $\{P_ikA_1P_k, \dots P_k A_kP_k\}$ is linearly independent. This completes the definition of the sequence $(P_n)$. Using selectivity, choose $P\in\mathcal F$ such that $P/n\leq P_n$, for $n\geq 1$. 

For $i, j, n\geq 1$ with $\max\{i^3,j^3\}>n$ and $i\neq j$, we have the following:

\[\mbox{Either } M_i>N_{i^3}>N_n\ \mbox{or } M_i>N_{j^3}>N_n.\] 

Hence

\[ \langle A_n Ph_i, Ph_j\rangle = \langle A_n P_i h_i, P_jh_j\rangle = \langle A_n h_{M_i}, h_{M_i}\rangle = 0.\]

\end{proof}


\section{Wide selective families and the generalized quantum Ramsey theorem}\label{main}


\begin{prop}\label{no diag}
Let $\mathcal F \subseteq\Pi_{\infty}(\mathcal H)$ be a wide, selective family. If $\mathcal V\subseteq B(\mathcal H)$ is an operator system such that there is no $P\in\mathcal F$ for which $P\mathcal V P$ is diagonalizable, then  either $\mathcal V$ has an infinite quantum clique in $\mathcal F$ or there exists $P\in\mathcal P$ and an operator system $\mathcal V'\subseteq\mathcal V$ such that the compression $P\mathcal V' P$ is infinite-dimensional and diagonalizable.
\end{prop}

\begin{proof}
 By Lemma \ref{one-two} there exists  $P\in \mathcal F$ such that for every $n$,  $P\mathcal{H}\ominus P\mathcal V P h_n\notin [\mathcal F]$. By identifying $\mathcal V$ with $P\mathcal V P$, we may assume that for every $n$, $P\mathcal H\ominus \mathcal V h_n\notin [\mathcal F]$. Therefore, by Definition \ref{sele def} (2), for every $n$, $\mathcal V h_n \in [\mathcal F]$. In particular, $\dim \mathcal V h_n = \infty$, for every $n\in\mathbb N$. Then, by Lemma \ref{diminf}, there exists a sequence $(A_n)$ in $\mathcal V$ and an orthonormal sequence $(x_n)$ in $\mathcal H$ such that for each $n$, $\langle A_n x_n, x_{n+1}\rangle = 1$ and  $\langle A_n x_i, x_j\rangle = 0$, for $i\neq j$ with $\max\{i,j\}>n+1$. 
 
 Suppose that for all $m\in\mathbb N$ there is a nonzero finite-rank operator in $\Span\{A_n : n\geq m\}$. Following the same procedure described in Lemmas 3.4, 3.5 and 2.3 of \cite{kks}, we obtain an isometry $V: \ell^2(\mathbb N)\rightarrow \mathcal H\oplus \mathcal H$ such that $V^*\mathcal V V$ is weakly dense in $B(\ell^2(\mathbb N))$. Define $T: \mathcal H\rightarrow \ell^2(\mathbb N)$ by $T(h) = V^*(h,0)$, for $h\in\mathcal H$. Since $\mathcal F$ is wide, by Lemma \ref{wide}, we can assume that the projection $P_{\ker(T)^{\perp}}$, onto $\ker(T)^{\perp}$, is an element of $\mathcal F$. Notice that  $T\mathcal V T^* = V^*(\mathcal V\oplus 0_{\mathcal H}) V$ is weakly dense in $B(\ell^2(\mathbb N))$. So $P_{\ker(T)^{\perp}}$ is an infinite quantum clique for $\mathcal V$ in  $\mathcal F$. 
 Otherwise, applying Lemma \ref{compress} above and Lemma 3.7 of \cite{kks} (truncating $\{A_n : n\in\mathbb N\}$, if necessary) we obtain $P\in\mathcal F$ and $\mathcal V'\subseteq\mathcal V$ such that the compression $P\mathcal V' P$ is infinite-dimensional and diagonalizable. This concludes the proof.

\end{proof}

	\begin{lem}[Infinite-dimensional case of Proposition 4.5 \cite{kks}, selective version]\label{clique}
 If $\mathcal V\subseteq B(\mathcal H)$ is an infinite-dimensional, diagonalizable (by the fixed orthonormal basis $\mathfrak B = (h_n)$) operator system, and  $\mathcal F \subseteq\Pi_{\infty}(\mathcal H)$ is a wide, selective family, then  $\mathcal V$ has an infinite quantum clique in $\mathcal F$.
\end{lem}
\begin{proof}

Let $(A_n)$ be a linearly independent sequence of operators in  $\mathcal V$, simultaneously diagonalized by $\mathfrak B$ . We're going to follow a procedure similar to the main idea in the proof of Lemmas 4.3 and 4.4 in \cite{kks}.
\medskip

Let $P_1\in\mathcal F$ be a projection such that $\langle A_1 P_1h_1, P_1h_1 \rangle\neq 0$ (note that this is possible because $\mathcal F$ is closed under finite changes), and let $B_1 = A_1$.

Take $B_2\in\Span\{A_1, A_2\}$ such that  $\langle B_2 P_1h_1, P_1h_1 \rangle  = 0$. Choose $P_2\leq P_1$ in $\mathcal F$ such that $P_2h_2\perp P_1h_1$ and  $\langle B_2 P_2h_2, P_2h_2 \rangle\neq 0$ (again, this is possible because $\mathcal F$ is closed under finite changes). We proceed inductively. For $k> 1$, suppose we have defined $P_k\in\mathcal F$, $B_k\in \mathcal V$ such that:

\begin{enumerate}
\item $P_k\leq P_{k-1}$.
\item $B_k\in\Span\{A_1, \dots, A_k\}$.
\item $\langle B_k P_kh_k, P_kh_k \rangle \neq 0$.
\item $\langle B_k P_ih_i, P_ih_i \rangle = 0$ for $1\leq i< k$.
\item $P_kh_k \in \mathcal H\ominus\Span\{P_i h_i : 1\leq i<k\}$.
\end{enumerate}

Take $B_{k+1} \in\Span\{A_1, A_2, \dots A_k\}$ such that $\langle B_{k+1} P_ih_i, P_ih_i\rangle = 0$ for $1\leq i \leq k$. Let $P_{k+1}\leq P_k$ in $\mathcal F$ such that $P_{k+1}h_{k+1} \in \mathcal H\ominus\Span\{P_i h_i : 1\leq i\leq k\}$ and $\langle B_{k+1} P_{k+1}h_{k+1}, P_{k+1}h_{k+1} \rangle\neq 0$. This completes the construction. 

By selectivity, there exists $P \in \mathcal F$ such that $P/k\leq P_k$, for every $k\geq 1$. Then,

\[\mbox{for } i<k,\  \langle B_k P h_i, P h_i \rangle = \langle B_k P_i h_i, P_i h_i \rangle =  0,\] 
 \[\mbox{and } \langle B_k P h_k, P h_k \rangle = \langle B_k P_k h_k, P_k h_k \rangle  \neq 0.\]

Now, for each $k$,  let $x_k = P h_k $ and let $C_k = B_k/\langle B_k x_k, x_k  \rangle$. So, we get the following:

\begin{enumerate}
\item The sequence $(PC_k P)$ is linearly independent in $P\mathcal V P$ and simultaneously diagonalizable by $(x_k)$.
\item For $k\geq 1$, $\langle C_k x_k, x_k \rangle = 1$.
\item For $k\geq 2$,  $\langle C_k x_i, x_i \rangle = 0$ for $1\leq i< k$.
\end{enumerate}

Without loss of generality, we can replace condition (1) by:

\begin{enumerate}
\item[{(1)'}] The sequence $(C_k )$ is linearly independent in $\mathcal V$ and simultaneously diagonalizable by $(x_k)$.
\end{enumerate}

Now, the proof can be completed by following the exact same steps in the proof of Lemma 4.4  of \cite{kks} applied to $\mathcal V$, $(C_k )$ and $(x_k)$, and using the fact that $\mathcal F$ is a wide:

Proceeding as in  Lemma 4.4  of \cite{kks}, we can show that there is an isometry $V: \ell^2(\mathbb N)\rightarrow \mathcal H\oplus \mathcal H$ such that $V^*\mathcal V V$ is weakly dense in $B(\ell^2(\mathbb N))$. Now, define $T: \mathcal H\rightarrow \ell^2(\mathbb N)$ by $T(h) = V^*(h,0)$, for $h\in\mathcal H$. By Lemma \ref{wide}, we can assume that  the projection $P_{\ker(T)^{\perp}}$, onto $\ker(T)^{\perp}$, is an element of $\mathcal F$. Notice that  $T\mathcal V T^* = V^*(\mathcal V\oplus 0_{\mathcal H}) V$ is weakly dense in $B(\ell^2(\mathbb N))$. So $P_{\ker(T)^{\perp}}$ is an infinite quantum clique for $\mathcal V$ in  $\mathcal F$. This concludes the proof.

\end{proof}

\begin{thm}\label{selective is Ramsey} (A generalization of the infinite quantum Ramsey theorem of \cite{kks}). Let $\mathcal F \subseteq\Pi_{\infty}(\mathcal H)$  be a family of infinite rank projections. If $\mathcal F$ is wide and selective, then it is Ramsey. 
\end{thm}

\begin{proof}

Let $\mathcal V\subseteq B(\mathcal H)$ be an infinite-dimensional operator system. We can compress $\mathcal V$ to a subspace with countably infinite dimension if necessary, so we will assume that the dimension of $\mathcal H$ is countably infinite. Let $\{h_n\}$ be an orthonormal basis for $\mathcal H$. If there exist $P\in \mathcal F$ such that $P\mathcal V P$ is diagonalizable, then we apply Lemma \ref{clique}, and finish. Otherwise, by Proposition \ref{no diag}, if $\mathcal V$ does not have an infinite quantum clique in $\mathcal F$, then there exists $P\in\mathcal P$ and an operator system $\mathcal V'\subseteq\mathcal V$ such that the compression $P\mathcal V' P$ is infinite-dimensional and diagonalizable. In the latter case, identifying $\mathcal V$ with $P\mathcal V' P$ and applying Proposition 3.8 again, $\mathcal V$ has an infinite quantum clique in $\mathcal F$. Thus, $\mathcal F$ is Ramsey. This completes the proof.

\end{proof}



\begin{thebibliography}{10}

\bibitem{bkk} Bény, C., Kempf, A., and Kribs, D.W., \textit{Quantum error correction on infinite- dimensional Hilbert spaces}, J. Math. Phys. \textbf{50} (2009), 062108. 
\bibitem{cdm} Calder\'on, D., Di Prisco, C., Mijares, J., \textit{Ramsey subsets of the space of infinite block sequences of vectors}, Fund. Math. \textbf{257} (2022), No. 2, 189--216.
\bibitem{dmn} Di Prisco, C., Mijares, J., Nieto, J.  \textit{Local Ramsey Theory: An abstract approach}. Math. Log. Quart. \textbf{63}, No. 5 (2017) 384--396.
 \bibitem{dmu} Di Prisco, C., Mijares, J. and Uzc\'ategui C., \textit{Ideal games and Ramsey sets},  Proc. Amer. Math. Soc. \textbf{140} (2012), No. 7, 2255--2265.
\bibitem{dsw} Duan R., Severini, S., and Winter, A., \textit{On Zero--Error Communication via Quantum
Channels in the Presence of Noiseless Feedback}, in IEEE Trans. on Inf. Theory, \textbf{62} (2016), No. 9, 5260--5277.
\bibitem{far} Farah, I.,  \textit{Semiselective Coideals}, Mathematika, \textbf{45} (1998), 79--103.
\bibitem{hmtu} Hru\v{s}ak M., Meza-Alcantara D., Thummel E. and Uzc\'ategui C., \textit{Ramsey type properties of ideals}, Ann. Pure Appl. Logic \textbf{168} (2017), No. 11, 2022--2049.
\bibitem{kks} Kennedy, Matthew; Kolomatski, Taras; Spivak, Daniel,  \textit{An infinite quantum Ramsey theorem}. J. Operator Theory \textbf{84} (2020), No.1, 49--65.
\bibitem{lov} Lovász, L.,  \textit{On the Shannon capacity graph}, IEEE Transactions on Information Theory, \textbf{25} (1979), No. 1,1--7.
\bibitem{mathias} Mathias, A. R, \textit{Happy families}, Ann. Math. Logic, \textbf{12} (1977), No. 1, 59--111.

 
\bibitem{goyoMLQ07} J.G. Mijares, \textit{A notion of selective ultrafilter corresponding to topological Ramsey spaces}, Mat. Log. Quart. \textbf{53}, No. 3 (2007), 255--267
 
\bibitem{ram} Ramsey, F.,  \textit{On a problem of formal logic}, Proc. London Math. Soc. Ser. 2, \textbf{30}, (1929), 264--286.
\bibitem{sole} Solecki S., \textit{Analytic ideals and their applications}, Ann. Pure Appl. Logic \textbf{99} (1999), No. 1--3, 51--72.
\bibitem{stah} Stahlke, D., \textit{Quantum Zero-Error Source-Channel Coding and Non-Commutative Graph Theory},  IEEE Trans. on Inf. Theo, \textbf{62}, No. 1, 2016.
\bibitem{tod} Todorcevic, S.,  \textit{Introduction to topological Ramsey spaces}, Princeton University Press, 2010.
\bibitem{weav17} Weaver, N., \textit{A “quantum” Ramsey theorem for operator systems}. Proc. Amer. Math. Soc. \textbf{145} (2017), 4595--4605.
\bibitem{weav21} Weaver,  N., \textit{Quantum graphs as quantum relations}, J. Geom. Anal. \textbf{31} (2021), No. 9, 9090--9112. 
\end{thebibliography}
\end{document}